\documentclass[12pt]{amsart}
\usepackage[margin=1in]{geometry}
\usepackage{amssymb,amsfonts,amsmath}
\usepackage{color}
\usepackage{soul}
\usepackage{enumerate}
\usepackage{mathrsfs}
\usepackage{hyperref}
\usepackage[capitalise]{cleveref}
\usepackage{constants}
\usepackage[T1]{fontenc}
\usepackage{bbm} 

\newtheorem{theorem}{Theorem}[section]
\newtheorem{corollary}[theorem]{Corollary}
\newtheorem{proposition}[theorem]{Proposition}
\newtheorem{lemma}[theorem]{Lemma}

\theoremstyle{definition}

\theoremstyle{remark}
\newtheorem*{remark}{Remark}

\numberwithin{equation}{section}

\crefname{figure}{Figure}{Figures}
\theoremstyle{plain}
\newtheorem*{theorem*}{Theorem}
\crefname{theorems}{Theorem}{Theorems}
\crefname{corollaries}{Corollary}{Corollaries}
\newtheorem*{corollary*}{Corollary}
\crefname{corollaries*}{Corollary}{Corollaries}
\crefname{lemma}{Lemma}{Lemmas}
\crefname{proposition}{Proposition}{Propositions}
\crefname{conjectures}{Conjecture}{Conjectures}
\newtheorem*{conjonjecture*}{Conjecture}
\crefname{conjonjectures*}{Conjecture}{Conjectures}
\crefname{definitions}{Definition}{Definitions}
\crefname{hypotheses}{Hypothesis}{Hypotheses}

\newcommand{\R}{\mathbb{R}}
\newcommand{\Q}{\mathbb{Q}}

\newcommand{\re}{\textup{Re}}
\newcommand{\im}{\textup{Im}}

\newcommand{\GL}{\mathrm{GL}}

\newcommand{\A}{\mathbb{A}}

\makeatletter
\DeclareFontFamily{U}  {MnSymbolF}{}
\DeclareSymbolFont{symbolsMN}{U}{MnSymbolF}{m}{n}
\SetSymbolFont{symbolsMN}{bold}{U}{MnSymbolF}{b}{n}
\DeclareFontShape{U}{MnSymbolF}{m}{n}{
    <-6>  MnSymbolF5
   <6-7>  MnSymbolF6
   <7-8>  MnSymbolF7
   <8-9>  MnSymbolF8
   <9-10> MnSymbolF9
  <10-12> MnSymbolF10
  <12->   MnSymbolF12}{}
\DeclareFontShape{U}{MnSymbolF}{b}{n}{
    <-6>  MnSymbolF-Bold5
   <6-7>  MnSymbolF-Bold6
   <7-8>  MnSymbolF-Bold7
   <8-9>  MnSymbolF-Bold8
   <9-10> MnSymbolF-Bold9
  <10-12> MnSymbolF-Bold10
  <12->   MnSymbolF-Bold12}{}
\DeclareMathSymbol{\tbigtimes}{\mathop}{symbolsMN}{2}
\newcommand*{\bigtimes}{%
  \DOTSB
  \tbigtimes
  \slimits@ 
}
\makeatother

\renewcommand{\tilde}{\widetilde}

\renewcommand{\bar}{\overline}

\renewcommand{\epsilon}{\varepsilon}

\renewcommand{\pmod}[1]{\, (\mathrm{mod} {\, #1})}

\newcommand{\Sym}{\mathrm{Sym}^m f}

\allowdisplaybreaks

\renewcommand{\pmod}[1]{\left(\mathrm{mod}\,\,#1\right)}

\newconstantfamily{abcon}{symbol=c}

\makeatletter
\let\@wraptoccontribs\wraptoccontribs
\makeatother

\title[]{Effective forms of the Sato--Tate conjecture}
\author{Jesse Thorner}
\address{Department of Mathematics, University of Illinois, Urbana, IL 61801}
\email{jesse.thorner@gmail.com}

\begin{document}

\begin{abstract}
We prove effective forms of the Sato--Tate conjecture for holomorphic cuspidal newforms which improve on the author's previous work (solo and joint with Lemke Oliver).  We also prove an effective form of the joint Sato--Tate distribution for two twist-inequivalent newforms.  Our results are unconditional because of recent work of Newton and Thorne.
\end{abstract}

\maketitle

\section{Introduction and statement of results}

Let $f(z)=\sum_{n=1}^{\infty}a_f(n)n^{\frac{k-1}{2}}e^{2\pi inz}\in S_k^{\mathrm{new}}(\Gamma_0(q))$ be a  cusp form (normalized so that $a_f(1)=1$) with trivial nebentypus.  If $f$ is also an eigenform for all of the Hecke operators and all of the Atkin--Lehner involutions $|_k W(q)$ and $|_k W(Q_p)$ for each prime $p|q$, then $f$ is a {\it newform} (see \cite[Section 2.5]{Ono}).  Throughout, we assume that $f$ does not have complex multiplication (CM), so there is no imaginary quadratic field $K$ such that for $p\nmid q$, $p$ is inert in $K$ if and only if $a_f(p)=0$.  Deligne's proof of the Weil conjectures implies that for each prime $p$, there exists an angle $\theta_p\in[0,\pi]$ such that $a_f(p) = 2\cos\theta_p$.  Serre's extension of the Sato--Tate conjecture \cite{MR1484415}, originally proposed for $f$ attached to a non-CM elliptic curve by modularity, asserts that if $f$ is non-CM, then the sequence $\{\theta_p\}$ is equidistributed in the interval $[0,\pi]$ with respect to the measure $d\mu_{\mathrm{ST}}:=(2/\pi)\sin^2\theta d\theta$.  Equivalently, one has
\begin{equation}
\label{eqn:Sato-Tate}
\pi_{f,I}(x):=\#\{p\leq x\colon \theta_p\in I,~p\nmid q\}\sim\mu_{\mathrm{ST}}(I)\pi(x)\qquad\textup{as $x\to\infty$,}
\end{equation}
where $\pi(x)=\#\{p\leq x\}$ and $I=[\alpha,\beta]\subseteq[0,\pi]$.  Barnet-Lamb, Geraghty, Harris, and Taylor \cite{BGHT} proved Serre's extension of the Sato--Tate conjecture.  See \cite{MR2383303} for an excellent overiew.

In \cite{JT1}, the author bounded the error term in \eqref{eqn:Sato-Tate} assuming that the $m$-th symmetric power lift $\Sym$ corresponds with a cuspidal automorphic representation of $\GL_{m+1}(\mathbb{A})$ for each $m\geq 1$, where $\mathbb{A}$ denotes the ring of adeles over $\Q$.  This implies that the symmetric power $L$-functions $L(s,\Sym)$ have an analytic continuation and functional equation of the expected type for all $m\geq 1$.  With this hypothesis, the author proved for fixed $f$ and $I$ that for all $\epsilon>0$, there exist constants $c_{\epsilon},C_{\epsilon}>0$ such that
\begin{equation}
\label{eqn:JT1}
|\pi_{f,I}(x)-\mu_{\mathrm{ST}}(I)\pi(x)|\leq c_{\epsilon}\pi(x)(\log x)^{-\frac{1}{8}+\epsilon},\qquad x>C_{\epsilon}.
\end{equation}

Until recently, it was known that $\Sym$ corresponds with a cuspidal automorphic representation of $\GL_{m+1}(\A)$ only for $m\leq 8$ \cite{CT,GJ,Kim,KS2}.  Recently, Newton and Thorne \cite{NT,NT2} proved that $\Sym$ corresponds with a cuspidal automorphic representation of $\GL_{m+1}(\A)$ for {\bf all} $m\geq 1$.  This inspired the author to improve the quality of \eqref{eqn:JT1} in the $x$-aspect and specify the uniformity with respect to $f$ and $I$.  In what follows, $c_1$, $c_2$, $c_3,\ldots$ denotes a sequence of positive, absolute, and effectively computable constants.

\begin{theorem}
	\label{thm:main_theorem}
	Let $f\in S_k^{\mathrm{new}}(\Gamma_0(q))$ be a non-CM newform as above, and let $I=[\alpha,\beta]\subseteq[0,\pi]$.  There exists a constant $\Cl[abcon]{thm1.1}$ such that
	\begin{equation}
	\label{eqn:st1}
	|\pi_{f,I}(x)-\mu_{\mathrm{ST}}(I)\pi(x)|\leq \Cr{thm1.1}\pi(x)\frac{\log(kq\log x)}{\sqrt{\log x}},\qquad x\geq 3.
	\end{equation}
	Furthermore, there exist constants $\Cl[abcon]{thm1.1_1}$, $\Cl[abcon]{thm1.1_2}$, and $\Cl[abcon]{Linnik}$ such that (with $\mu=\mu_{\mathrm{ST}}(I)$)
	\begin{equation}
	\label{eqn:st2}
	\Cr{thm1.1_1}\mu_{\mathrm{ST}}(I)\pi(x)\leq \pi_{f,I}(x)\leq \Cr{thm1.1_2}\mu_{\mathrm{ST}}(I)\pi(x),\qquad x\geq (kq/\mu)^{\Cr{Linnik}\log(e/\mu)/\mu^2}.
	\end{equation}
\end{theorem}
\begin{remark}
The error term in \eqref{eqn:st1} is uniform enough to accommodate some small-scale equidistribution for the sequence $\{\theta_p\}$.   For instance, if $x\geq e^{kq}$ and $F(x)$ is a monotonically increasing function with $\lim_{x\to\infty}F(x)=\infty$, then the sequence $\{\theta_p\}$ is equidistributed with respect to $\mu_{\mathrm{ST}}$ in intervals with $\mu_{\mathrm{ST}}(I)$ as small as $(\log\log x)F(x)/\sqrt{\log x}$.
\end{remark}

\begin{remark}
The proof can be modified for when the nebentypus character $\psi\pmod{q}$ of $f$ is nontrivial.  However, if $\omega$ is a root of unity with $\omega^2$ in the image of $\psi$, then one must restrict consideration to the primes $p$ such that $\psi(p) = \omega^2$.  Then $a_f(p)/\omega\in\R$, and we can define $\theta_p\in[0,\pi]$ by $a_f(p)=2\omega\cos\theta_p$.  The ensuing changes to our proofs are not purely cosmetic.  In particular, the $m$-dependence in \cref{cor:ZFR} would become worse since the cuspidal automorphic representation of $\mathrm{GL}_2(\mathbb{A}_F)$ corresponding with $f$ is no longer guaranteed to be self-dual.  This directly affects the power of $\log x$ that is saved in \cref{thm:main_theorem}.
\end{remark}
\begin{remark}
The bounds in \eqref{eqn:st2} improve work of Lemke Oliver and the author \cite[Thm 1.6]{RJLOT}.	If $I$ is fixed, then as $f$ varies, \eqref{eqn:st2} gives an upper bound on $\pi_{f,I}(x)$ commensurate with the Brun--Titchmarsh theorem as well as an upper bound on the least $p\nmid q$ such that $\theta_p\in I$ commensurate with Linnik's bound on the least prime in an arithmetic progression \cite{Linnik}.
\end{remark}

Even though the error term in asymptotic in \cref{thm:main_theorem} saves less than a full power of $\log x$ over $\pi(x)$, some arithmetically significant consequences still follow from \cref{thm:main_theorem}.  For example, Luca, Radziwi{\l}{\l}, and Shparlinski \cite[Thm 1.1]{MR3893309} proved that the inequality $|a_f(n)|\leq (\log n)^{-\frac{1}{2}+o(1)}$ (where $o(1)$ denotes a quantity, possibly depending on $f$, which tends to zero as $n\to\infty$) holds for a density one subset of integers $n$.  This improves on a standard argument which achieves the same bound with $\log 2$ replacing $-\frac{1}{2}$.  One might ask whether the exponent $-\frac{1}{2}$ might be lowered any further for a density one subset of $n$.  In the same paper, Luca, Radziwi{\l}{\l}, and Shparlinski \cite[Cor 1.5]{MR3893309}  proved that  \cref{thm:main_theorem} suffices to show that if $v\in\R$ is fixed, then
\begin{equation}
\label{eqn:maks}
\lim_{x\to\infty}\frac{\displaystyle\#\left\{n\leq x\colon a_f(n)\neq 0,~\frac{\log|a_f(n)|+\frac{1}{2}\log\log n}{\sqrt{(\frac{1}{2}+\frac{\pi^2}{12})\log\log n}}\geq v\right\}}{\#\{n\leq x\colon a_f(n)\neq 0\}}=\frac{1}{\sqrt{2\pi}}\int_{v}^{\infty}e^{-u^2/2}du.
\end{equation}
Thus the exponent $-\frac{1}{2}$ cannot be lowered any further for a density one subset of $n$.  More recently, Klurman and Mangerel \cite{2020arXiv200903225K} used \cref{thm:main_theorem} to prove a multidimensional version of \eqref{eqn:maks}, which enabled them to prove that if $N\geq 1$, $A_f(n)=a_f(n)n^{\frac{k-1}{2}}$, and $b_1,\ldots,b_N$ are distinct nonnegative integers such that $A_f(b_j)\neq 0$ for all $1\leq j\leq N$, then
\begin{equation}
\label{eqn:oleksiy}
\lim_{x\to\infty}\frac{\{n\geq x\colon 0<|A_f(n+b_1)|<|A_f(n+b_2)|<\cdots<|A_f(n+b_N)|\}}{\#\{n\leq x\colon A_f(n+b_1)\cdots A_f(n+b_N)\neq0\}}=\frac{1}{N!}.
\end{equation}
In both \eqref{eqn:maks} and \eqref{eqn:oleksiy}, if $x$ is sufficiently large, then there exists a constant $c_f>0$ such that the counting function in the denominator is bounded below by $c_f x$.  This lower bound is a direct consequence of the fact that there exists a constant $\delta>0$ such that the density of primes $p\leq x$ such that $a_f(p)=0$ is at most $(\log x)^{-\delta}$ \cite{Serre,TZ2}.



For $i=1,2$, let $f_i\in S_{k_i}^{\mathrm{new}}(q_i)$ be a non-CM newform, let $\{\theta_{p}^{(i)}\}$ be the sequence of angles in the Sato--Tate conjecture for $f_i$, and let $\mathbf{1}_{I_i}$ be the indicator function of the interval $I_i=[\alpha_i,\beta_i]\subseteq[0,\pi]$.  Suppose that there exists no primitive character $\chi$ satisfying the property that $f_1=f_2\otimes\chi$, in which case $f_1$ and $f_2$ are not twist-equivalent.  We denote this by $f_1\not\sim f_2$.  If $f_1\not\sim f_2$, then it is natural to ask whether the Sato--Tate distributions for $f_1$ and $f_2$ are independent, as quantified by the proposed existence of the asymptotic
\begin{equation}
\label{eqn:pi2}
\pi_{f_1,f_2,I_1,I_2}(x):=\sum_{\substack{p\leq x \\  p\nmid q_1 q_2}}\mathbf{1}_{I_1}(\theta_{p}^{(1)})\mathbf{1}_{I_2}(\theta_{p}^{(2)})\sim \mu_{\mathrm{ST}}(I_1)\mu_{\mathrm{ST}}(I_2)\pi(x)\qquad\textup{as $x\to\infty$,}
\end{equation}
where $I_1,I_2\subseteq[0,\pi]$ are two intervals.  This question was posed independently by Katz and Mazur for $f_i$ corresponding to elliptic curves by modularity.  Harris \cite{Harris} proved the asymptotic \eqref{eqn:pi2} for non-CM newforms associated to pairs of twist-inequivalent elliptic curves.  See Wong \cite{Wong} for a generalization to all pairs of non-CM newforms and even pairs of Hilbert modular forms over totally real fields.  The work of Newton and Thorne \cite{NT,NT2} and the ideas leading to \cref{thm:main_theorem} permit us to effectively quantify this independence.

\begin{theorem}
\label{thm:joint}
	For $i=1,2$, let $f_i\in S_{k_i}(\Gamma_0(q_i))$ be a newform as in \cref{thm:main_theorem}, and let $I_i=[\alpha_i,\beta_i]\subseteq[0,\pi]$.  If $f_1\not\sim f_2$, then there exists a constant $\Cl[abcon]{thm1.4}$ such that
	\[
	|\pi_{f_1,f_2,I_1,I_2}(x) - \mu_{\mathrm{ST}}(I_1)\mu_{\mathrm{ST}}(I_2)\pi(x)|\leq \Cr{thm1.4}\pi(x)\frac{\log(k_1 q_1 k_2 q_2\log\log x)}{\sqrt{\log\log x}},\qquad x\geq 16.
	\]
\end{theorem}

\begin{remark}
One could establish a much stronger ineffective error term, where the ineffectivity arises from the possible existence of a Landau--Siegel zero (see \cref{sec:zeros,sec:Sym}).  See Molteni \cite{Molteni} (and also his PhD thesis) for a discussion on how to bound such zeros in our setting.  As one sees in \cref{sec:Sym}, Landau--Siegel zeros do not plague the error term in \cref{thm:main_theorem}.
\end{remark}

If one assumes the generalized Riemann hypothesis (GRH) for each symmetric power $L$-function $L(s,\Sym)$, then \cref{thm:main_theorem} improves as follows.

\begin{theorem}
\label{thm:GRH}
Let $f\in S_k(\Gamma_0(q))$ and $I$ be as in \cref{thm:main_theorem}, and let $\mu=\mu_{\mathrm{ST}}(I)$.  If $L(s,\Sym)$ satisfies GRH for all $m\geq 0$, then there exists a constant $\Cl[abcon]{thm1.3}$ such that
	\[
	|\pi_{f,I}(x)-\mu_{\mathrm{ST}}(I)\pi(x)|\leq \Cr{thm1.3}x^{\frac{3}{4}}\frac{\log(kqx)}{\log x},\qquad x\geq 3.
	\]
	Also, there exist a constant $\Cl[abcon]{ZLinnike3}$ and a prime $p\nmid q$ which satisfies $\theta_p\in I$ and $p\leq \Cr{ZLinnike3}\frac{1}{\mu^4}(\log\frac{kq}{\mu})^2$.
\end{theorem}
\begin{proof}
	Rouse and the author \cite{RT} proved the first result for squarefree levels.  Chen, Park, and Swaminathan \cite{CPS} proved the second result when $f$ corresponds with an elliptic curve by modularity, again for squarefree levels.  Their proof extends to other $f$ with little additional effort using \cref{thm:automorphy} below.  The work in \cite{CPS,RT} assumed the automorphy of the symmetric powers of $f$, which is now known unconditionally \cite{NT,NT2}.
\end{proof}

\begin{remark}
The authors in \cite{CPS,RT} assumed a squarefree level in order to produce strong explicit bounds on $\Cr{thm1.3}$ and $\Cr{ZLinnike3}$.  The orders of magnitude do not change when $q$ is not squarefree, but the task of obtaining strong bounds on the implied constants seems difficult.
\end{remark}

\begin{remark}
An error term of size $c_{f,\epsilon}x^{\frac{1}{2}+\epsilon}$ is expected for all fixed $\epsilon>0$.  See \cite[Thm 1.4]{MR3860468} for some compelling on-average results in this direction.
\end{remark}

If one assumes GRH for the Rankin--Selberg $L$-functions associated to the tensor products of $\mathrm{Sym}^{m_1} f_1$ and $\mathrm{Sym}^{m_2} f_2$ for all $m_1,m_2\geq 0$, then \cref{thm:joint} improves as follows.

\begin{theorem}
\label{thm:joint_GRH}
For $i=1,2$, let $f_i\in S_{k_i}^{\mathrm{new}}(\Gamma_0(q_i))$ and $I_i$ be as in \cref{thm:joint}.  If $f_1\not\sim f_2$ and the Rankin--Selberg $L$-functions $L(s,\mathrm{Sym}^{m_1} f_1\times \mathrm{Sym}^{m_2} f_2)$ satisfy GRH for all integers $m_1,m_2\geq 0$, then there exists a constant $\Cl[abcon]{thm1.5}$ such that
	\[
	|\pi_{f_1,f_2,I_1,I_2}(x) - \mu_{\mathrm{ST}}(I_1)\mu_{\mathrm{ST}}(I_2)\pi(x)|\leq \Cr{thm1.5}x^{\frac{5}{6}}\log(k_1 q_1 k_2 q_2 x)^{\frac{1}{3}},\qquad x\geq 3.
	\]
\end{theorem}

\begin{proof}
	Bucur and Kedlaya \cite{BK} proved this when $f_1$ and $f_2$ correspond with elliptic curves by modularity.  One can extend their proof to other $f_1$ and $f_2$ with little additional effort.
\end{proof}

%
%


\subsection*{Acknowledgements} I thank Maksym Radziwi{\l}{\l}, Jeremy Rouse, Jack Thorne, John Voight, and Peng-Jie Wong for helpful discussions and the anonymous referees.

\section{Proof of \cref{thm:main_theorem,thm:joint}}

In this section, we reduce the proofs of \cref{thm:main_theorem,thm:joint} to the task of proving a uniform version of the prime number theorem for certain classes of $L$-functions.  We use the Vinogradov notation $F\ll G$ to denote the existence of an absolute and effectively computable constant $c>0$ (not necessarily the same in each occurrence) such that $|F|\leq c|G|$ in the range indicated.  We write $F=G+O(H)$ to denote that $|F-G|\ll H$.

\subsection{Proof of \cref{thm:main_theorem}}

Let $I=[\alpha,\beta]\subseteq[0,\pi]$ be an interval.  Let $\mathbf{1}_I$ be the indicator function of the interval $I$, and define $\pi_{f,I}(x)$ as in  \eqref{eqn:Sato-Tate}.  Let
\[
U_m(\cos\theta_p)=\frac{\sin((m+1)\theta_p)}{\sin\theta_p}=\sum_{j=0}^m e^{(2j-m)i\theta_p}
\]
be the $m$-th Chebyshev polynomial of the second type.  These polynomials form an orthonormal basis for $L^2([0,\pi],\mu_{\mathrm{ST}})$ with respect to the usual inner product $\langle f,g\rangle=\int_0^{\pi}f(\theta)g(\theta)d\mu_{\mathrm{ST}}$.

Montgomery \cite[Chapter 1]{MR1297543} used the work of Beurling and Selberg to efficiently majorize and minorize the indicator function of a subinterval of $[0,1]$ using carefully constructed trigonometric polynomials.  One performs a suitable change of variables to handle subintervals of $[0,\pi]$ and a change of basis to express the trigonometric polynomials in terms of the Chebyshev polynomials $U_m(\cos\theta)$ (see \cite[Section 3]{RT}). For any integer $M\geq 3$, we find that
\begin{equation}
\label{eqn:ET1}
|\pi_{f,I}(x)-\mu_{\mathrm{ST}}(I)\pi(x)|\ll \frac{\pi(x)}{M}+\sum_{m=1}^M\frac{1}{m}\Big|\sum_{\substack{p\leq x \\ p\nmid q}}U_m(\cos\theta_p)\Big|,\qquad x\geq 3.
\end{equation}
Note that $\pi(x)\sim x/\log x$ by the prime number theorem.  By partial summation, we have
\begin{equation}
\label{eqn:partial_sums_1}
\sum_{\substack{p\leq x \\ p\nmid q}}U_m(\cos\theta_p) = \frac{\theta_{f,m}(x)}{\log x}-\int_2^x\frac{\theta_{f,m}(t)}{t(\log t)^2}dt,\qquad \theta_{f,m}(x):=\sum_{\substack{p\leq x \\ p\nmid q}}U_m(\cos\theta_p)\log p.
\end{equation}
\begin{proposition}
\label{prop:PNT_Sym}
	Let $f$ be as in \cref{thm:main_theorem}.  There exist constants $\Cl[abcon]{LFZDE_u}$ (suitably large) and $\Cl[abcon]{ZFR_u}$ and $\Cl[abcon]{M_range_u}$ (suitably small) such that if $1\leq m\leq \Cr{M_range_u}\sqrt{\log x}/\sqrt{\log(kq\log x)}$, then
	\begin{align*}
	|\theta_{f,m}(x)|\ll m^2 x^{1-\frac{1}{\Cr{LFZDE_u}m}}+m^2 x \Big(\exp\Big[-\Cr{ZFR_u}\frac{\log x}{m^2\log(kqm)}\Big]+\exp\Big[-\Cr{ZFR_u}\frac{\sqrt{\log x}}{\sqrt{m}}\Big]\Big).
	\end{align*}
\end{proposition}

By \eqref{eqn:ET1}, \eqref{eqn:partial_sums_1}, and \cref{prop:PNT_Sym}, if $3\leq M\leq \Cr{M_range_u} \sqrt{\log x}/\sqrt{\log(kq \log x)}$, then
\begin{equation}
\label{eqn:discrep1}
\begin{aligned}
&|\pi_{f,I}(x)-\mu_{\mathrm{ST}}(I)\pi(x)|\\
&\ll \pi(x)\Big(\frac{1}{M}+\sum_{m=1}^M m\Big(x^{-\frac{1}{2\Cr{LFZDE_u}m}}+\exp\Big[-\frac{\Cr{ZFR_u}}{2}\frac{\log x}{m^2\log(kqm)}\Big]+\exp\Big[-\frac{\Cr{ZFR_u}}{2}\frac{\sqrt{\log x}}{\sqrt{m}}\Big]\Big)\Big).
\end{aligned}
\end{equation}

\begin{proof}[Proof of \cref{thm:main_theorem}]
	Write $\mu=\mu_{ST}(I)$.  There exists a suitably large constant $\Cl[abcon]{range_const_1}$ such that if $x\geq \exp(\Cr{range_const_1}(\frac{1}{\mu}\log\frac{kq}{\mu})^2)$, then the bound \eqref{eqn:st1} follows from \eqref{eqn:discrep1} by choosing $\Cl[abcon]{Mconst}>0$ to be a sufficiently small constant and $M=\lceil\Cr{Mconst}\sqrt{\log x}/\log(kq \log x)\rceil\geq 3$.  For all remaining values of $x$, \eqref{eqn:st1} is trivial.  The bounds in \eqref{eqn:st2} follow from \eqref{eqn:discrep1} by choosing $\Cl[abcon]{Linnik2}>0$ to be a sufficiently large constant, $M=\lceil \Cr{Linnik2}/\mu_{\mathrm{ST}}(I)\rceil$, and $x\geq (kqM)^{\Cr{Linnik2}M^2\log(M)/\Cr{ZFR_u}}$.
\end{proof}

\subsection{Proof of \cref{thm:joint}}

For $i=1,2$, let $\mu_i=\mu_{\mathrm{ST}}(I_i)$ and let $f_i$ be as in \cref{thm:joint} with Sato--Tate angles $\{\theta_p^{(i)}\}$.  Let $\pi_{f_1,f_2,I_2,I_2}(x)$ be as in \eqref{eqn:pi2}.  Cochrane \cite{Cochrane} carried out a version of Montgomery's analysis which constructs trigonometric polynomials which efficiently majorize and minorize the indicator function of $R_1\times R_2$, where $R_i=[a_i,b_i]\subseteq[0,1]$.  Upon performing a change of variables and a change of basis similar to \cite[Section 3]{RT} to express these polynomials in terms of Chebyshev polynomials, we find for any integer $M\geq 3$ that
	\begin{equation}
	\label{eqn:ET2}
	\begin{aligned}
	&|\pi_{f_1,f_2,I_1,I_2}(x)-\mu_{\mathrm{ST}}(I_1) \mu_{\mathrm{ST}}(I_2) \pi(x)|\\
	&\ll \frac{\pi(x)}{M}+\sum_{\substack{0\leq m_1,m_2\leq M \\ m_1 m_2\neq 0}}\frac{1}{(m_1+1)(m_2+1)}\Big|\sum_{\substack{p\leq x \\ p\nmid q_1 q_2}}U_{m_1}(\cos\theta_p^{(1)})U_{m_2}(\cos\theta_p^{(2)})\Big|,
	\end{aligned}
	\quad x\geq 3.
	\end{equation}
	By partial summation, we have
\begin{equation}
\label{eqn:partial_sums_2}
\sum_{\substack{p\leq x \\ p\nmid q_1 q_2}}U_{m_1}(\cos\theta_p^{(1)})U_{m_2}(\cos\theta_p^{(2)}) = \frac{\theta_{f_1,f_2,m_1,m_2}(x)}{\log x}-\int_2^x\frac{\theta_{f_1,f_2,m_1,m_2}(t)}{t(\log t)^2}dt,
\end{equation}
where
\begin{equation}
\label{eqn:theta_2}
\theta_{f_1,f_2,m_1,m_2}(x):=\sum_{\substack{p\leq x \\ p\nmid q_1 q_2}}U_{m_1}(\cos\theta_p^{(1)})U_{m_2}(\cos\theta_p^{(2)})\log p.
\end{equation}
\begin{proposition}
\label{prop:PNT_Sym_2}
	For $i=1,2$, let $f_i$ be as in \cref{thm:joint}.  There exist constants $\Cl[abcon]{LFZDE_v}$ and $\Cl[abcon]{Siegel2_v}$ (suitably large) and $\Cl[abcon]{ZFR_v}$, $\Cl[abcon]{M_range_v}$, and $\Cl[abcon]{Siegel_v}$ (suitably small) such that if
	\begin{equation}
	\label{eqn:M-rng}
	1\leq m_1,m_2\leq M\leq \Cr{M_range_v} \sqrt{\log\log x}/\log(k_1 q_1 k_2 q_2\log\log x),
	\end{equation}
	then
	\begin{equation*}
	\begin{aligned}
		|\theta_{f_1,f_2,m_1,m_2}(x)|&\ll x^{1-\frac{\Cr{Siegel_v}}{(k_1 q_1 k_2 q_2 M)^{\Cr{Siegel2_v}M^2}}}+ (m_1 m_2)^2 x^{1-\frac{1}{\Cr{LFZDE_v}M^2}}\\
	&+(m_1 m_2)^2 x \Big(\exp\Big[-\Cr{ZFR_v}\frac{\log x}{M^2\log(k_1 q_1 k_2 q_2 M)}\Big]+\exp\Big[-\Cr{ZFR_v}\frac{\sqrt{\log x}}{M}\Big]\Big).
	\end{aligned}
	\end{equation*}
\end{proposition}

Note that $\theta_{f_1,f_2,0,m_2}(x)=\theta_{f_2,m_2}(x)$ and $\theta_{f_1,f_2,m_1,0}(x)=\theta_{f_1,m_1}(x)$.  Therefore, by \cref{prop:PNT_Sym}, \eqref{eqn:ET2}, \eqref{eqn:partial_sums_2}, \eqref{eqn:theta_2}, and \cref{prop:PNT_Sym_2}, we find that if \eqref{eqn:M-rng} holds, then
\begin{multline}
	\label{eqn:discrep2}
		\Big|\pi_{f_1,f_2,I_1,I_2}(x)-\mu_{\mathrm{ST}}(I_1) \mu_{\mathrm{ST}}(I_2)\pi(x)\Big|\\
		\ll \pi(x)\Big\{\frac{1}{M}+M^2 x^{-\frac{\Cr{Siegel_v}}{2(k_1 q_1 k_2 q_2 M)^{\Cr{Siegel2_v}M^2}}}+M^4 x^{-\frac{1}{2\Cr{LFZDE_v}M^2}}\\
		+M^4 \Big(\exp\Big[-\Cr{ZFR_v}\frac{\log x}{2M^2\log(k_1 q_1 k_2 q_2 M)}\Big]+\exp\Big[-\Cr{ZFR_v}\frac{\sqrt{\log x}}{2M}\Big]\Big)\Big\}.
	\end{multline}
\begin{proof}[Proof of \cref{thm:joint}]
Write $\mu = \mu_{\mathrm{ST}}(I_1)\mu_{\mathrm{ST}}(I_2)$.  There exists a suitably large constant $\Cl[abcon]{c_range_2}$ such that if $x\geq \exp\exp(\Cr{c_range_2}(\frac{1}{\mu}\log\frac{k_1 q_1 k_2 q_2}{\mu})^2)$, then the theorem follows from \eqref{eqn:discrep2} by choosing $\Cl[abcon]{Mconst2}>0$ to be a sufficiently small constant and $M=\lceil\Cr{Mconst2}\frac{\sqrt{\log\log x}}{\log(k_1 q_1 k_2 q_2\log\log x)}\rceil\geq 3$.  For all remaining values of $x$, the theorem is trivial.
\end{proof}

\subsection{Outline for proofs of \cref{prop:PNT_Sym,prop:PNT_Sym_2}}

\cref{prop:PNT_Sym,prop:PNT_Sym_2} will follow from a sufficiently uniform unconditional prime number theorem for the $L$-functions associated to symmetric powers of the cuspidal automorphic resepresentations associated to newforms and the Rankin--Selberg convolutions of said symmetric powers.  To prove such prime number theorems, we first review well-known properties of $L$-functions in \cref{sec:properties}.  We prove zero-free regions and log-free zero density estimates for $L$-functions which satisfy the generalized Ramanujan conjecture in \cref{sec:zeros}.  In \cref{sec:PNT}, we use the results in \cref{sec:zeros} to prove a highly uniform prime number theorem for $L$-functions satisfying the generalized Ramanujan conjecture.  Finally, in \cref{sec:Sym}, we use the aforementioned work of Newton and Thorne \cite{NT,NT2} to show how $L$-functions of symmetric powers and their convolutions fit into the framework of \cref{sec:PNT}.  We then prove  \cref{prop:PNT_Sym,prop:PNT_Sym_2}.

Our work requires careful attention to the degree dependence in several $L$-functions estimates.  Other problems in analytic number theory typically do not require such care, but in our setting, the degree aspect of our estimates is the aspect that matters most since it directly determines the quality of our results.  In order to make the degree dependencies in the necessary estimates as strong as we can, we will refine the work in \cite[Appendix]{Humphries} and \cite{ST} to prove the necessary zero-free regions and log-free zero density estimates.

\section{Properties of $L$-functions}
\label{sec:properties}

\subsection{Standard $L$-functions}

Let $\A$ be the ring of adeles of $\Q$, and $\mathfrak{F}_{m}$ be the set of cuspidal automorphic representations of $\GL_m(\A)$ with unitary central character, which assume to be normalized so that it is trivial on the diagonally embedded copy of the positive real numbers.  Given $\pi\in\mathfrak{F}_m$, let $\tilde{\pi}$ be the the representation which is contragredient to $\pi$, and let $q_{\pi}\geq 1$ be the conductor of $\pi$.  Write the finite part of $\pi$ as a tensor product $\otimes_p \pi_p$ of local representations over primes $p$.  For each $p$, there exist Satake parameters $\alpha_{1,\pi}(p),\ldots,\alpha_{m,\pi}(p)\in\mathbb{C}$ such that the local $L$-function $L(s,\pi_p)$ is given by
\[
L(s,\pi_p)=\prod_{j=1}^m \Big(1-\frac{\alpha_{j,\pi}(p)}{p^s}\Big)^{-1}=\sum_{j=0}^{\infty}\frac{a_{\pi}(p^j)}{p^{js}}.
\]
We have $\alpha_{j,\pi}(p)\neq 0$ for all $j$ when $p\nmid q_{\pi}$, and some of the $\alpha_{j,\pi}(p)$ might equal zero when $p|q_{\pi}$.  The standard $L$-function $L(s,\pi)$ attached to $\pi$ is
\[
L(s,\pi)=\prod_p L(s,\pi_p)=\sum_{n=1}^{\infty}\frac{a_{\pi}(n)}{n^s},
\]
which converges absolutely for $\re(s)>1$.

The gamma factor corresponding to the infinite place of $\Q$ is given by
\[
L(s,\pi_{\infty})=\prod_{j=1}^{m}\Gamma_{\R}(s+\mu_{\pi}(j)),\qquad \Gamma_{\R}(s)=\pi^{-s/2}\Gamma(s/2),
\]
where $\mu_{\pi}(1),\ldots,\mu_{\pi}(n)\in\mathbb{C}$ are the Langlands parameters.  The bounds
\begin{equation}
\label{eqn:LRS}
|\alpha_{j,\pi}(p)|\leq \theta_m,\qquad \re(\mu_{\pi}(j))\geq -\theta_m
\end{equation}
hold for some $0\leq \theta_m\leq \frac{1}{2}-\frac{1}{m^2+1}$.  The generalized Ramanujan conjecture and generalized Selberg eigenvalue conjectures assert that the above inequalities hold with $\theta_m=0$.

Let $r_{\pi}\in\{0,1\}$ be the order of the pole of $L(s,\pi)$ at $s=1$, where $r_{\pi}=1$ if and only if $\pi$ is the trivial representation $\mathbbm{1}$ of $\GL_1(\A)$ whose $L$-function is the Riemann zeta function $\zeta(s)$. The function $\Lambda(s,\pi)=(s(s-1))^{r_{\pi}}q_{\pi}^{s/2}L(s,\pi)L(s,\pi_{\infty})$ is entire of order one.  There exists a complex number $W(\pi)$ of modulus one such that $\Lambda(s,\pi)=W(\pi)\Lambda(1-s,\tilde{\pi})$, where $\tilde{\pi}\in\mathfrak{F}_m$ is the contragredient representation.  We have the equalities of sets
\[
\{\alpha_{j,\tilde{\pi}}(p)\}=\{\overline{\alpha_{j,\pi}(p)}\},\qquad\{\mu_{\tilde{\pi}}(j)\}=\{\overline{\mu_{\pi}(j)}\},
\]
and $q_{\tilde{\pi}}=q_{\pi}$.  We define the analytic conductor $C(\pi)$ by
\[
C(\pi,t) = q_{\pi}\prod_{j=1}^n (1+|\mu_{\pi}(j)+it|),\qquad C(\pi)=C(\pi,0).
\]

Define the numbers $\Lambda_{\pi}(n)$ by the Dirichlet series identity
\[
\sum_{n=1}^{\infty}\frac{\Lambda_{\pi}(n)}{n^s} =-\frac{L'}{L}(s,\pi) = \sum_p \sum_{\ell=1}^{\infty}\frac{\sum_{j=1}^m \alpha_{j,\pi}(p)^{\ell}\log p}{p^{\ell s}},\qquad \re(s)>1.
\]
It was proved in the discussion following \cite[Lem 2.3]{ST} that for all $\eta>0$, we have
	\begin{equation}
		\label{eqn:RS2}
		\sum_{n=1}^{\infty}\frac{|\Lambda_{\pi}(n)|}{n^{1+\eta}}\leq \frac{1}{\eta}+m\log C(\pi)+O(m^2).
	\end{equation}

\subsection{Rankin--Selberg $L$-functions}

Let $\pi\in\mathfrak{F}_m$ and $\pi'\in\mathfrak{F}_{m'}$.  For each prime $p$, we let
\[
L(s,\pi_p\times\pi_p')=\prod_{j=1}^{m}\prod_{j'=1}^{m'}\Big(1-\frac{\alpha_{j,j',\pi\times\pi'}(p)}{p^s}\Big)^{-1}=1+\sum_{j=1}^{\infty}\frac{a_{\pi\times\pi'}(p^j)}{p^{js}}
\]
for suitable complex numbers $\alpha_{j,j',\pi\times\pi'}(p)$.  If $p\nmid q_{\pi}q_{\pi'}$, then we have the equality of sets $\{\alpha_{j,j',\pi\times\pi'}(p)\} = \{\alpha_{j,\pi}(p)\alpha_{j',\pi'}(p)\}$.  A complete description of $\alpha_{j,j',\pi\times\pi'}(p)$ is given in \cite[Appendix]{ST}.  From these Satake parameters, one defines the Rankin--Selberg $L$-function
\[
L(s,\pi\times\pi')=\prod_{p}L(s,\pi_p\times\pi_p')=\sum_{n=1}^{\infty}\frac{a_{\pi\times\pi'}(n)}{n^{s}}
\]
associated to the tensor product $\pi\otimes\pi'$, which converges absolutely for $\re(s)>1$.  We write $q_{\pi\times\pi'}$ for the conductor of $\pi\otimes\pi'$.  Bushnell and Henniart \cite{BH} proved that $q_{\pi\times\pi'}|q_{\pi}^{n'}q_{\pi'}^n$.

The gamma factor corresponding to the infinite place of $\Q$ is given by
\[
L(s,\pi_{\infty}\times\pi_{\infty}')=\prod_{j=1}^{m}\prod_{j'=1}^{m'}\Gamma_{\R}(s+\mu_{\pi\times\pi'}(j,j'))
\]
for suitable complex numbers $\mu_{\pi\times\pi'}(j,j')$.  If both $\pi$ and $\pi'$ are unramified at the infinite place of $\Q$, then we have the equality of sets $\{\mu_{\pi\times\pi'}(j,j')\} = \{\mu_{\pi}(j)+\mu_{\pi'}(j')\}$.  A complete description of the numbers can be found in \cite[Proof of Lemma 2.1]{ST}.  From the explicit descriptions of the numbers $\alpha_{j,j',\pi\times\pi'}(p)$ and $\mu_{\pi\times\pi'}(j,j')$, we find that
\begin{equation}
\label{eqn:LRS2}
|\alpha_{j,j',\pi\times\pi'}(p)|\leq \theta_m+\theta_{m'},\qquad \re(\mu_{\pi\times\pi'}(j,j'))\geq -\theta_m-\theta_{m'}.
\end{equation}

Let $r_{\pi\times\pi'}\in\{0,1\}$ be the order of the pole of $L(s,\pi\times\pi')$ at $s=1$.  We have that $r_{\pi\times\pi'}=1$ if and only if $\pi'=\tilde{\pi}$. The function $\Lambda(s,\pi\times\pi')=(s(s-1))^{r_{\pi\times\pi'}}q_{\pi\times\pi'}^{s/2}L(s,\pi\times\pi')L(s,\pi_{\infty}\times\pi_{\infty}')$ is entire of order one.  There exists a complex number $W(\pi\times\pi')$ of modulus one such that $\Lambda(s,\pi\times\pi')=W(\pi\times\pi')\Lambda(1-s,\tilde{\pi}\times\tilde{\pi}')$.  We define the analytic conductor $C(\pi\times\pi')$ to be
\[
C(\pi\times\pi',t) = q_{\pi\times\pi'}\prod_{j=1}^m \prod_{j'=1}^{m'} (1+|\mu_{\pi\times\pi'}(j,j')+it|),\qquad C(\pi\times\pi')=C(\pi\times\pi',0).
\]
The combined work of Bushnell and Henniart \cite{BH} and Brumley \cite[Appendix]{Humphries} proves that
\begin{equation}
\label{eqn:BH}
C(\pi\times\pi',t)\ll C(\pi\times\pi')(1+|t|)^{m'm},\qquad C(\pi\times\pi')\leq e^{O(m'm)}C(\pi)^{m'}C(\pi')^m.
\end{equation}

Define the numbers $\Lambda_{\pi\times\pi'}(n)$ by the Dirichlet series identity
\[
\sum_{n=1}^{\infty}\frac{\Lambda_{\pi\times\pi'}(n)}{n^s}=-\frac{L'}{L}(s,\pi\times\pi')=\sum_{p}\sum_{\ell=1}^{\infty}\frac{\sum_{j=1}^{m} \sum_{j'=1}^{m'}\alpha_{j,j',\pi\times\pi'}(p)^\ell\log p}{p^{\ell s}},\qquad \re(s)>1.
\]
It was proved in the discussion following \cite[Lem 2.3]{ST} that for all $\eta>0$, we have
	\begin{equation}
		\label{eqn:RS_22}
		\sum_{n=1}^{\infty}\frac{|\Lambda_{\pi\times\pi'}(n)|}{n^{1+\eta}}\leq \frac{1}{\eta}+m'm\log C(\pi\times\pi')+O((m'm)^2).
	\end{equation}

\subsection{Isobaric automorphic representations}

Let $d\geq 1$ be an integer; let $m_1,\ldots, m_d$ be positive integers; let $r = \sum_{i=1}^d m_i$; let $t_1,\ldots,t_d\in\R$; and let $\pi_i\in\mathfrak{F}_{m_i}$ with $1\leq i\leq d$.  Consider the isobaric automorphic representation $\Pi$ of $\GL_{r}(\A)$ given by the isobaric sum
\[
\Pi=\pi_1\otimes|\det|^{it_1}\boxplus\cdots\boxplus\pi_d\otimes|\det|^{it_d}.
\]
The $L$-function associated to $\Pi$ is $L(s,\Pi)=\prod_{j=1}^d L(s+it_j,\pi_j)$, with analytic conductor
\[
C(\Pi,t)=\prod_{j=1}^d C(\pi_j,t+t_j),\qquad C(\Pi)=C(\Pi,0).
\]

Let $d'\geq 1$ be an integer; let $m_1',\ldots, m_{d'}'$ be positive integers; let $r' = \sum_{i'=1}^{d'} m_{i'}'$; let $t_1',\ldots,t_d'\in\R$; and let $\pi_{i'}'\in\mathfrak{F}_{m_{i'}'}$ with $1\leq i'\leq d'$.  Consider the isobaric automorphic representation $\Pi'$ of $\GL_{r'}(\A)$ given by the isobaric sum $\Pi'=\pi_1'\otimes|\det|^{it_1'}\boxplus\cdots\boxplus\pi_{d'}\otimes|\det|^{it_{d'}'}$.  The Rankin--Selberg $L$-function associated to $\Pi\otimes\Pi'$ is
\[
L(s,\Pi\times\Pi')=\prod_{j=1}^d \prod_{j'=1}^{d'}L(s+it_j+it_{j'}',\pi_j\times\pi_{j'}'),
\]
and its analytic conductor is
\[
C(\Pi\times\Pi',t)=\prod_{j=1}^d \prod_{j'=1}^{d'}C(\pi_j\times\pi_{j'}',t+t_j+t_{j'}'),\qquad C(\Pi\times\Pi')=C(\Pi\times\Pi',0).
\]
	
\section{Zeros of $L$-functions}
\label{sec:zeros}

We require two results on the distribution of zeros of standard $L$-functions and Rankin--Selberg $L$-functions.  First, we require a standard zero-free region.  We present a modification to the work of Brumley \cite[Appendix]{Humphries} which will improve the degree dependence.  We will use this improved degree dependence to prove our second result, namely, the log-free zero density estimate of Soundararajan and the author \cite{ST} with improved degree dependence.

\subsection{Zero-free regions}

A proof of the following proposition is sketched in \cite[Lem 3.1]{RJLOT}.  We give a complete proof here.

\begin{proposition}
	\label{prop:GHL}
	Let $\Pi$ be an isobaric automorphic representation of $\GL_{r}(\A)$.  If $L(s,\Pi\times\tilde{\Pi})$ has a pole of order $r_{\Pi\times\tilde{\Pi}}\geq 1$ at $s=1$, then $L(1,\Pi\times\tilde{\Pi})\neq 0$, and there exists a constant $\Cl[abcon]{GHL}>0$ such that $L(s,\Pi\times\tilde{\Pi})$ has at most $r_{\Pi\times\tilde{\Pi}}$ real zeros in the interval
	\begin{equation}
	\label{eqn:claimed_GHL}
	s\geq 1-\frac{\Cr{GHL}}{(r_{\Pi\times\tilde{\Pi}}+1)\log C(\Pi\times\tilde{\Pi})}.
	\end{equation}
\end{proposition}
\begin{remark}
Our proof removes the extraneous factor of $d$ in the denominator in \cite[Lem 5.9]{IK} when $f=\Pi\times\tilde{\Pi}$.  While this may seem like a small improvement, the quality of our main theorems depends heavily on it.
\end{remark}
\begin{proof}
	By proceeding as in the proof of \cite[Equations 5.28 and 5.37]{IK}, we find that
	\[
	\sum_{\substack{1-\Cr{GHL}/\log C(\Pi\times\tilde{\Pi})<\beta\leq 1 \\ L(\beta,\Pi\times\tilde{\Pi})=0 }}\frac{1}{\sigma-\beta}<\frac{r_{\Pi\times\tilde{\Pi}}}{\sigma-1}+\re\Big(\frac{L'}{L}(\sigma,\Pi\times\tilde{\Pi})\Big)+O(\log C(\Pi\times\tilde{\Pi})),
	\]
	where $s=\sigma\geq 1$.  Define $\Lambda_{\Pi\times\tilde{\Pi}}(n)$ by the Dirichlet series identity
	\[
	\sum_{n=1}^{\infty}\frac{\Lambda_{\Pi\times\tilde{\Pi}}(n)}{n^s}=-\frac{L'}{L}(s,\Pi\times\tilde{\Pi}).
	\]
	The work in \cite[Section A.2]{ST} leading up to Equation A.9 shows that $\Lambda_{\Pi\times\tilde{\Pi}}(n)\geq0$ for all $n\geq 1$ (even if $n$ shares a prime factor with a conductor of one of the constituents of $\Pi$).  Thus by nonnegativity, we find that
	\begin{equation}
	\label{eqn:5.37}
	\sum_{\substack{1-\Cr{GHL}/\log C(\Pi\times\tilde{\Pi})<\beta\leq 1 \\ L(\beta,\Pi\times\tilde{\Pi})=0 }}\frac{1}{\sigma-\beta}<\frac{r_{\Pi\times\tilde{\Pi}}}{\sigma-1}+O(\log C(\Pi\times\tilde{\Pi})).
	\end{equation}
	Let $N$ be the number of real zeros in the sum on the left hand side of \eqref{eqn:5.37}.  We choose $\sigma=1+2\Cr{GHL}/\log C(\Pi\times\tilde{\Pi})$ and conclude that
	\[
	\frac{N\log C(\Pi\times\tilde{\Pi})}{2\Cr{GHL}+\frac{\Cr{GHL}}{r_{\Pi\times\tilde{\Pi}}+1}}<\Big(\frac{r_{\Pi\times\tilde{\Pi}}}{2\Cr{GHL}}+O(1)\Big)\log C(\Pi\times\tilde{\Pi}).
	\]
	This implies that $N<r_{\Pi\times\tilde{\Pi}}+\frac{r_{\Pi\times\tilde{\Pi}}}{2(r_{\Pi\times\tilde{\Pi}}+1)}+O(\Cr{GHL})$, so $N\leq r_{\Pi\times\tilde{\Pi}}$ when $\Cr{GHL}$ is small enough.
	
	The possibility that $L(1,\Pi\times\tilde{\Pi})=0$ is ruled out by \cite[Thm A.1]{Lapid}.
\end{proof}

\begin{corollary}
	\label{cor:ZFR}
	Let $\pi\in\mathfrak{F}_m$ and $\pi'\in\mathfrak{F}_{m'}$.  Suppose that both $\pi$ and $\pi'$ are self-dual (that is, $\pi=\tilde{\pi}$ and $\pi'=\tilde{\pi}'$).  There exists a constant $\Cl[abcon]{ZFR}>0$ such that the following results hold.
	\begin{enumerate}
		\item $L(s,\pi)\neq 0$ in the region
		\[
	\re(s)\geq 1-\frac{\Cr{ZFR}}{m\log(C(\pi)(3+|\im(s)|))}
	\]
	apart from at most one zero.  If the exceptional zero exists, then it is real and simple.
	\item $L(s,\pi\times\pi')\neq 0$ in the region
	\begin{equation}
	\label{eqn:claimed_ZFR_2}
	\re(s)\geq 1-\frac{\Cr{ZFR}}{(m+m')\log(C(\pi)C(\pi')(3+|\im(s)|)^{\min\{m,m'\}})}
	\end{equation}
	apart from at most one zero.  If the exceptional zero exists, then it is real and simple.
	\end{enumerate}
	\end{corollary}
	\begin{remark}
	This improves the denominator $(m+m')^3\log(C(\pi)C(\pi')(3+|\im(s)|)^{m})$ in \cite[Thm A.1]{Humphries} when $F=\Q$ and both $\pi$ and $\pi'$ are self-dual.  A similar improvement holds when $\pi$ and $\pi'$ are defined over number fields.
	\end{remark}

	\begin{proof}
	We prove the second part; the proof of the first part is the same once we choose $\pi'=\mathbbm{1}$.  Without loss, suppose that $m'\leq m$.  First, let $\gamma\neq 0$, and suppose to the contrary that $\rho=\beta+i\gamma$ is a zero in the region \eqref{eqn:claimed_ZFR_2}.  We apply \cref{prop:GHL} to the choice of $\Pi=\pi'\otimes|\det|^{i\gamma}\boxplus\tilde{\pi}'\otimes|\det|^{-i\gamma}\boxplus\pi$, in which case (since $\pi$ and $\pi'$ are self-dual)
	\begin{align*}
L(s,\Pi\times\tilde{\Pi})&=L(s,\pi\times\tilde{\pi})L(s,\pi'\times\tilde{\pi}')^2L(s+it,\pi\times\pi')^2 L(s-it,\pi\times\pi')^2\\
&\times L(s+2i\gamma,\pi'\times\tilde{\pi}')
L(s-2i\gamma,\pi'\times\tilde{\pi}').
\end{align*}
  Since $\pi$ and $\pi'$ are self-dual, it follows that $\rho$ is a zero of $L(s,\pi\times\pi')$ if and only if $\bar{\rho}$ is.  Thus if $\rho$ is a zero of $L(s,\pi)$, then $L(s,\Pi\times\tilde{\Pi})$ has a real zero at $s=\beta$ of order 4 in the region \eqref{eqn:claimed_GHL}.  This contradicts \cref{prop:GHL} since $r_{\Pi\times\tilde{\Pi}}=3$ when $\gamma\neq 0$.  The desired result follows from the bound $\log C(\Pi\times\tilde{\Pi})\ll (m+m')\log(C(\pi)C(\pi')(3+|\gamma|)^{m'})$, which holds via \eqref{eqn:BH}.

Second, suppose that $\gamma=0$.  The same arguments as when $\gamma\neq 0$ hold, except that now $r_{\Pi\times\tilde{\Pi}}=5$.  Since the presence of a single zero of $L(s,\pi\times\pi')$ in the claimed region contributes a zero of order 4 to $L(s,\Pi\times\tilde{\Pi})$, we must conclude that if a real zero of $L(s,\pi\times\pi')$ exists in the region \eqref{eqn:claimed_ZFR_2}, then such a zero must be simple.
\end{proof}

We also record a bound on the exceptional zero in Part (2) of \cref{cor:ZFR} (if it exists).

\begin{lemma}
\label{lem:Siegel}
	There exists a constant $\Cl[abcon]{Siegel_1}>0$ such that the exceptional zero in Part (2) of \cref{cor:ZFR} when $\pi\neq\pi'$ is bounded by $1-\Cr{Siegel_1}(C(\pi)C(\pi'))^{-m-m'}$.
\end{lemma}
\begin{proof}
	This is an immediate consequence of \cite[Thm A.1]{Lapid}.
\end{proof}

\subsection{Log-free zero density estimates}

Our log-free zero density estimates are as follows.

\begin{proposition}
\label{prop:LFZDE}
	Let $\pi\in\mathfrak{F}_m$ and $\pi'\in\mathfrak{F}_{m'}$.  There exists a constant $\Cl[abcon]{LFZDE}>2$ such that the following are true.
	\begin{enumerate}
		\item Suppose that $\pi$ satisfies GRC at all primes $p\nmid q_{\pi}$.  If $0\leq\sigma\leq 1$ and $T\geq 1$, then
		\[
		N_{\pi}(\sigma,T):=\#\{\rho=\beta+i\gamma\colon \beta\geq\sigma,~|\gamma|\leq T,~L(\rho,\pi)=0\}\ll m^2(C(\pi)T)^{\Cr{LFZDE}m(1-\sigma)}.
		\]
		\item If $0\leq\sigma\leq 1$, $T\geq 1$, $m,m'\leq M$, and $\pi$ and $\pi'$ satisfy GRC at all $p\nmid q_{\pi}q_{\pi'}$,  then
		{\small\[
		N_{\pi\times\pi'}(\sigma,T):=\#\{\rho=\beta+i\gamma\colon \beta\geq\sigma,~|\gamma|\leq T,~L(\rho,\pi\times\pi')=0\}\ll M^4 ((C(\pi)C(\pi'))^{M} T)^{\Cr{LFZDE}M^2(1-\sigma)}.
		\]}
	\end{enumerate}
\end{proposition}

For the sake of brevity, we will only prove Part (1).  There are no structural differences in the proof of Part (2) except that $m$ is replaced by $m'm\leq M^2$ and $C(\pi)$ is replaced by $C(\pi)^{O(m')}C(\pi')^{O(m)}$ (which is an upper bound for $C(\pi\times\pi')$).  Our proof runs parallel to that of \cite[Thm 1.2]{ST}, so we only point out the key differences.  We begin with some adjustments to the lemmas in \cite{ST} which will improve the degree dependence.

\begin{lemma}
	\label{lem:02}
	If $t\in\R$ and $0<\eta\leq 1$, then
	\[
	\sum_{\rho}\frac{1+\eta-\beta}{|1+\eta+it-\rho|^2}\leq 2m\log C(\pi)+m\log(2+|t|)+\frac{2}{\eta}+O(m^2)
	\]
	so that $\#\{\rho\colon |\rho-(1+it)|\leq\eta\}\leq 10m\eta\log C(\pi)+5m\eta\log(2+|t|)+O(m^2\eta+1)$.
\end{lemma}
\begin{proof}
	The proof proceeds just as in \cite[Lem 3.1]{ST}, but we use \eqref{eqn:RS2} instead of \cite[Equation 1.9]{ST}, and we use the fact that $r_{\pi}\leq 1$ in our case.
\end{proof}

\begin{lemma}
	\label{lem:03}
	Let $T\geq 1$, and let $\tau\in\R$ satisfy $200\eta\leq |\tau|\leq T$.  If $\Cr{ZFR}/(m\log(C(\pi)T))\leq\eta\leq (200m)^{-1}$ and $s=1+\eta+i\tau$, then
	\[
	\frac{(-1)^k}{k!}\Big(\frac{L'}{L}(s,\pi)\Big)^{(k)}=\sum_{\substack{\rho \\ |s-\rho|\leq 200\eta}}\frac{1}{(s-\rho)^{k+1}}+O\Big(\frac{m\log(C(\pi)T)}{(200\eta)^k}\Big).
	\]
\end{lemma}
\begin{proof}
	The proof is the same as that of \cite[Equation 4.1]{ST} with three changes.  First, we have that $r_{\pi}\leq 1$.  Second, we widen the range of $\eta$ all the way to the edge of the zero-free region in \cref{cor:ZFR}.  Finally, we use \cref{lem:02} instead of \cite[Lem 3.1]{ST}.
\end{proof}

\begin{lemma}
\label{lem:04}
	Let $T\geq 1$, $200\eta\leq |\tau|\leq T$, and $\Cr{ZFR}/(m\log(C(\pi)T))<\eta\leq (200m)^{-1}$.  Let $K>\lceil 2000m\eta\log(C(\pi)T)+O(m^2\eta+1)\rceil$ and  $s=1+\eta+i\tau$.  If $L(s,\pi)$ has a zero $\rho_0$ satisfying $|\rho_0-(1+i\tau)|\leq \eta$, then
	\[
	\Big|\sum_{\substack{\rho \\ |s-\rho|\leq200\eta}}\frac{1}{(s-\rho)^{k+1}}\Big|\geq \Big(\frac{1}{100\eta}\Big)^{k+1}.
	\]
\end{lemma}
\begin{proof}
	The proof is the same as that of \cite[Lem 4.2]{ST}, but we use \cref{lem:02,lem:03} instead of \cite[Equations 3.6 and 4.1]{ST}.
\end{proof}

\begin{lemma}
	\label{lem:05}
	Let $\eta$ and $\tau$ be real numbers satisfying $\Cr{ZFR}/(m\log(C(\pi)T))<\eta\leq (200m)^{-1}$ and $200\eta\leq |\tau|\leq T$.  Let $K\geq 1$ be an integer, and let $N_0=\exp(K/(300\eta))$ and $N_1=\exp((40K)/\eta)$.  Let $s=1+\eta+i\tau$,  If $K\leq k\leq 2K$, then
	\[
	\Big|\frac{\eta^{k+1}}{k!}\Big(\frac{L'}{L}(s,\pi)\Big)^{(k)}\Big|\leq\eta^2\int_{N_0}^{N_1}\Big|\frac{\Lambda_{\pi}(n)}{n^{1+i\tau}}\Big|\frac{du}{u}+O\Big(\frac{m\eta\log(C(\pi)T)}{(110)^k}\Big).
	\]
\end{lemma}
\begin{proof}
	The proof is the same as that of \cite[Lem 4.3]{ST} except that it incorporates \eqref{eqn:RS2} instead of \cite[Equation 19]{ST} as well as our wider range of $\eta$.
\end{proof}

\begin{proof}[Proof of \cref{prop:LFZDE}]
	We reiterate that we have only given the details for $L(s,\pi)$, but the details for $L(s,\pi\times\pi')$ run parallel apart from bounding $C(\pi\times\pi')$.  Compare with \cite{ST}.
	
	Let $\eta,\tau\in\R$ satisfy $\Cr{ZFR}/(m\log(C(\pi)T))<\eta\leq (200m)^{-1}$ and $200\eta\leq |\tau|\leq T$.  Let
	\begin{align*}
	K= 2000\eta m\log(C(\pi)T)+O(m^2\eta+1),\quad N_0=\exp(K/(300\eta)),\quad N_1=\exp(40K/\eta), 
	\end{align*}
	where the implied constant in our definition of  $K$ is sufficiently large.  We proceed as in the proof of \cite[Thm 1.2]{ST}, but we use \cref{lem:04,lem:05} instead of \cite[Lem 4.2 and 4.3]{ST}.  We conclude that
	\begin{align*}
		&\#\{\rho=\beta+i\gamma\colon \beta\geq 1-\eta/2,~|\gamma|\leq T\}\\
		&\ll 102^{4K}\eta^2 m\log(C(\pi)T)T^2\int_{N_0/e}^{N_1}\Big|\sum_{x<n\leq xe^{1/T}}|\Lambda_{\pi}(n)|\Big|^2\frac{dx}{x^3}+\eta m\log(C(\pi)T)\\
		&\ll 102^{4K}K\eta T^2\int_{N_0/e}^{N_1}\Big|\sum_{x<n\leq xe^{1/T}}|\Lambda_{\pi}(n)|\Big|^2\frac{dx}{x^3}+K
	\end{align*}
	By hypothesis, we have the bounds $|\alpha_{j,\pi}(p)|\leq 1$ for $p\nmid q_{\pi}$ and $|\alpha_{j,\pi}(p)|\leq p^{1-1/m}$ for $p|q_{\pi}$.  Since $q_{\pi}$ has $O(\log q_{\pi})$ distinct prime divisors, for $x$ in the range $[N_0/e,N_1]$, we have
	\begin{equation}
	\label{eqn:prime_power_contribution_LFZDE}
	\begin{aligned}
	\sum_{x<n\leq xe^{1/T}}|\Lambda_{\pi}(n)|&\leq m\sum_{\substack{x<n\leq xe^{1/T} \\ \gcd(n,N_{\pi})=1}}\Lambda(n)+m\sum_{p|N_{\pi}}\log p\sum_{\frac{\log x}{\log p}<\ell\leq \frac{\log x}{\log p}+\frac{1}{T\log p}}p^{\ell(1-1/m)}\\
	&\ll mx(1+x^{-\frac{1}{2m}}\log q_{\pi})/T\ll mx/T,
	\end{aligned}
	\end{equation}
	hence
	\begin{align*}
		&\#\{\rho=\beta+i\gamma\colon \beta\geq 1-\eta/2,~|\gamma|\leq T\}\\
		&\ll m^2 102^{4K}K\eta T^2\int_{N_0/e}^{N_1}\Big|\frac{x}{T}\Big|^2\frac{dx}{x^3}+K\ll m^2 102^{4K}K\eta T^2\cdot\frac{K}{\eta T^2}+K\ll  m^2 105^{4K}.
	\end{align*}
	
	If we write $\eta=2(1-\sigma)$, then it follows that
	\[
	N_{\pi}(\sigma,T)\ll m^2(e^{O(m)}C(\pi)T)^{10^5 m (1-\sigma)},\qquad 1-\frac{1}{400m}\leq\sigma< 1-\frac{\Cr{ZFR}}{2m\log(C(\pi)T)}.
	\]
	If $\sigma\geq 1-\Cr{ZFR}/(2m\log(C(\pi)T))$, then $N_{\pi}(\sigma,T)\leq 1$ per \cref{cor:ZFR}.  Our proposition is trivial for $\sigma<1-1/(400m)$ since there are $\ll mT\log(C(\pi)T)$ zeros with $0<\beta<1$ and $|\gamma|\leq T$; see \cite[Thm 5.8]{IK}.  To finish the proof, we apply the bound $e^{O(m)}\leq C(\pi)^{O(1)}$.
\end{proof}

\section{An uniform prime number theorem}
\label{sec:PNT}

We will prove a uniform version of the prime number theorem for the $L$-functions $L(s,\pi)$ and $L(s,\pi\times\pi')$ which satisfy the hypotheses of \cref{cor:ZFR,prop:LFZDE}.  The use of \cref{prop:LFZDE} helps us to improve the range of $x$ (which important for \eqref{eqn:st2}) and simplify certain aspects of the proof.  We define the weighted prime counting functions
\[
\theta_{\pi}(x):=\sum_{\substack{p\leq x \\ p\nmid q_{\pi}}}a_{\pi}(p)\log p,\qquad \theta_{\pi\times\pi'}(x):=\sum_{\substack{p\leq x \\ p\nmid q_{\pi}q_{\pi'}}}a_{\pi\times\pi'}(p)\log p.
\]

\begin{proposition}
\label{prop:PNT}
Let $\pi\in\mathfrak{F}_m$ and $\pi'\in\mathfrak{F}_{m'}$.  Suppose that $\pi$ and $\pi'$ satisfy the hypotheses of \cref{cor:ZFR,prop:LFZDE}.
\begin{enumerate}
	\item Suppose that either $\mu_{\pi}(j)=0$ or $\re(\mu_{\pi}(j))\geq\frac{1}{2}$ for each $j$.  If $2\leq C(\pi)^m\leq x^{1/(36\Cr{LFZDE})}$ and $32\Cr{LFZDE}m x^{-1/(32\Cr{LFZDE}m)}<\frac{1}{4}$, then
	\begin{align*}
	&\Big|\theta_{\pi}(x) - r_{\pi}x+\frac{x^{\beta_1}}{\beta_1}\Big|\\
	&\ll m^2 x^{1-\frac{1}{32\Cr{LFZDE}m}}+m^2 x\Big(\exp\Big[-\frac{\Cr{ZFR}\log x}{2m\log C(\pi)}\Big]+\exp\Big[-\frac{\sqrt{\Cr{ZFR}\log x}}{2\sqrt{m}}\Big]\Big).
	\end{align*}
	We omit the $\beta_1$ term if the exceptional zero in \cref{cor:ZFR} (Part 1) does not exist.
	\item Let $m,m'\leq M$.  Suppose that either $\mu_{\pi\times\pi'}(j,j')=0$ or $\re(\mu_{\pi\times\pi'}(j,j'))\geq\frac{1}{2}$.  If $2\leq (C(\pi)C(\pi'))^{M^3}\leq x^{1/(36\Cr{LFZDE})}$ and $32\Cr{LFZDE}M^2 x^{-1/(32\Cr{LFZDE}M^2)}<\frac{1}{4}$, then
	\begin{align*}
	&\Big|\theta_{\pi\times\pi'}(x) - r_{\pi\times\pi'}x+\frac{x^{\beta_1}}{\beta_1}\Big|\\
	&\ll (m'm)^2 x^{1-\frac{1}{32\Cr{LFZDE}M^2}}+(m'm)^2 x\Big(\exp\Big[-\frac{\Cr{ZFR}\log x}{4M\log (C(\pi)C(\pi'))}\Big]+\exp\Big[-\frac{\sqrt{\Cr{ZFR}\log x}}{2M}\Big]\Big).
	\end{align*}
	We omit the $\beta_1$ term if the exceptional zero in \cref{cor:ZFR} (Part 2) does not exist.
\end{enumerate}
\end{proposition}

As with our proof of \cref{prop:LFZDE}, we will only prove Part (1) of \cref{prop:PNT} since the proof of Part (2) runs entirely parallel apart from an application of \eqref{eqn:BH} to bound  $C(\pi\times\pi')$.  To begin our work toward \cref{prop:PNT}, we introduce a carefully chosen smooth weight for sums over prime powers.

\begin{lemma}
\label{lem:WeightChoice}

Choose $x \geq 3$, $\epsilon \in (0,1/4)$, and an integer $\ell \geq 1$.  Define $A = \epsilon/(2 \ell \log x)$.  There exists a continuous function $\phi(t)  = \phi(t; x, \ell, \epsilon)$ which satisfies the following properties:
\begin{enumerate}
	\item $0 \leq \phi(t) \leq 1$ for all $t \in \R$, and $\phi(t) \equiv 1$ for $\tfrac{1}{2} \leq t \leq 1$.
	\item The support of $\phi$ is contained in the interval $[\tfrac{1}{2} - \frac{\epsilon}{\log x}, 1 +  \frac{\epsilon}{\log x}]$. 
	\item Its Laplace transform $\Phi(z) = \int_{\R} \phi(t) e^{-zt}dt$ is entire and is given by
			\begin{equation}	
				\Phi(z) = e^{-(1+ 2\ell A)z} \cdot \Big( \frac{1-e^{(\frac{1}{2}+2\ell A)z}}{-z} \Big) \Big( \frac{1-e^{2Az}}{-2Az} \Big)^{\ell}.
				\label{eqn:WeightLaplace}
			\end{equation}
	\item Let $s = \sigma + i t, \sigma > 0, t \in \R$ and $\alpha$ be any real number satisfying $0 \leq \alpha \leq \ell$. Then
	\[
	|\Phi(-s\log x)|\leq 
		\displaystyle\frac{e^{\sigma \epsilon} x^{\sigma}}{|s| \log x} \cdot \big( 1 + x^{-\sigma/2} \big) \cdot  \Big( \frac{2\ell}{\epsilon|s|} \Big)^{\alpha}. 
	\]
	Moreover, $|\Phi(-s\log x)| \leq e^{\sigma \epsilon} x^{\sigma}$ and $\frac{1}{2} < \Phi(0) < \frac{3}{4}$.
	\item If $\frac{3}{4}<\sigma\leq 1$, $x\geq 10$, and $\delta_{\pi,1},r_{\pi}\in\{0,1\}$, then
	\[
	r_{\pi}\Phi(-\log x)\log x-\delta_{\pi,1} \Phi(-\sigma\log x)\log x=\Big(r_{\pi}x-\delta_{\pi,1}\frac{x^{\sigma}}{\sigma}\Big)(1+O(\epsilon))+O(\sqrt{x}).
	\]
	\item Let $s = -\tfrac{1}{4}+it$ with $t \in \R$. Then
	\[
	|\Phi(-s\log x)| \ll  \frac{x^{-1/4}}{\log x} \Big( \frac{2\ell}{\epsilon}\Big)^{\ell} (\tfrac{1}{16}+t^2)^{-\ell/2}.
	\]
\end{enumerate}
\end{lemma}
\begin{proof}
	This is \cite[Lem 2.2]{TZ3}, except that item (5) is slightly more general and in item (6), we have $\re(s)=-\frac{1}{4}$ instead of $-\frac{1}{2}$.  The proofs proceed in exactly the same way.
\end{proof}

Define
\[
\psi_{\pi}(x)=\psi_{\pi}(x,\phi) = \sum_{n=1}^{\infty}\phi\Big(\frac{\log n}{\log x}\Big)\Lambda_{\pi}(n).
\]
The next lemma shows that $\psi_{\pi}(x)$ closely approximates $\theta_{\pi}(x)$.
\begin{lemma}
\label{lem:smoothing}
	If $\ell\geq 2$ is an integer, $x$ satisfies the hypotheses of \cref{prop:PNT_Sym}, and $\epsilon\in(x^{-\frac{1}{2}},\frac{1}{4})$, then $\theta_{\pi}(x)=\psi_{\pi}(x)+O(mx^{1-\frac{1}{3m}}+\epsilon m x)$.
\end{lemma}
\begin{proof}
By \cref{lem:WeightChoice}, we have
\[
\psi_{\pi}(x) = \sum_{\sqrt{x}<n\leq x}\Lambda_{\pi}(x)+O\Big(\sum_{\sqrt{x}e^{-\epsilon}<n\leq\sqrt{x}}|\Lambda_{\pi}(n)|+\sum_{x<n\leq xe^{\epsilon}}|\Lambda_{\pi}(n)|\Big).
\]
Note that by \eqref{eqn:LRS} and the Brun--Titchmarsh theorem, we have
\begin{align*}
\sum_{\sqrt{x}<n\leq x}\Lambda_{\pi}(n)&= \sum_{\substack{n\leq x \\ \gcd(n,q_{\pi})=1}}\Lambda_{\pi}(n)+\sum_{\substack{\sqrt{x}<n\leq x \\ \gcd(n,q_{\pi})>1}}\Lambda_{\pi}(n)-\sum_{\substack{n\leq\sqrt{x} \\ \gcd(n,q_{\pi})=1}}\Lambda_{\pi}(n)\\
&=\sum_{\substack{n\leq x \\ \gcd(n,q_{\pi})=1}}\Lambda_{\pi}(n)+O\Big(m\sum_{p|q_{\pi}}\log p\sum_{\frac{\log x}{2\log p}<\ell\leq \frac{\log x}{\log p}}p^{(1-\frac{1}{m})\ell}+m\sum_{n\leq\sqrt{x}}\Lambda(n)\Big)\\
&=\sum_{\substack{n\leq x \\ \gcd(n,q_{\pi})=1}}\Lambda_{\pi}(n)+O(mx^{1-\frac{1}{2m}}\log q_{\pi})=\sum_{\substack{n\leq x \\ \gcd(n,q_{\pi})=1}}\Lambda_{\pi}(n)+O(mx^{1-\frac{1}{3m}}).
\end{align*}
It follows from a calculation similar to the one in \eqref{eqn:prime_power_contribution_LFZDE} that
\[
\sum_{\sqrt{x}e^{-\epsilon}<n\leq\sqrt{x}}|\Lambda_{\pi}(n)|+\sum_{x<n\leq xe^{\epsilon}}|\Lambda_{\pi}(n)|\ll \epsilon m x.
\]
The lemma now follows from the fact that $a_{\pi}(p)\log p=\Lambda_{\pi}(p)$ for all primes $p$, hence
\[
\Big|\theta_{\pi}(x)-\sum_{\substack{n\leq x \\ \gcd(n,q_{\pi})=1}}\Lambda_{\pi}(n)\Big|=\Big|\sum_{\ell=2}^{\infty}\sum_{\substack{p\nmid q_{\pi} \\ p^{\ell}\leq x}}\Lambda_{\pi}(p^{\ell})\Big|\leq m\sum_{\ell=2}^{\infty}\sum_{\substack{p^{\ell}\leq x \\ p\nmid q_{\pi}}}\log p\ll m\sqrt{x}.
\]
\end{proof}

By Mellin inversion, we have
\[
\psi_{\pi}(x)=\frac{\log x}{2\pi i}\int_{2-i\infty}^{2+i\infty}-\frac{L'}{L}(s,\pi)\Phi(-s\log x)ds.
\]
\begin{lemma}
\label{lem:log_deriv_bound}
If $\re(s)=-\frac{1}{4}$, then $-\frac{L'}{L}(s,\pi)\ll \log C(\pi)+m\log(|\im(s)|+3)$.
\end{lemma}
\begin{proof}
	With the hypotheses of \cref{prop:PNT}, this is follows from \cite[Prop 5.7]{IK}.
\end{proof}

We begin our proof of \cref{prop:PNT} by shifting the contour to the line $\re(s)=-\frac{1}{4}$, accumulating contributions from the residues at the nontrivial zeros of $L(s,\pi)$ and a trivial zero at $s=0$ of order $O(m)$ with residue $O(1)$.  We bound the shifted contour integral on the line $\re(s)=-\frac{1}{4}$ using \cref{lem:WeightChoice,lem:log_deriv_bound} and conclude that
\begin{align}
\label{eqn:001}
\psi_{\pi}(x)&=r_{\pi}\Phi(-\log x)\log x-\delta_{\pi,1}\Phi(-\beta_1\log x)\log x-(\log x)\sum_{\rho\neq\beta_1}\Phi(-\rho\log x)\notag\\
&+O\Big(\frac{(2\ell/\epsilon)^{\ell}\log C(\pi)}{x^{\frac{1}{4}}}+m\Big)\notag\\
&=r_{\pi}x-\delta_{\pi,1}\frac{x^{\beta_1}}{\beta_1}-(\log x)\sum_{\rho\neq\beta_1}\Phi(-\rho\log x)+O\Big(\frac{(2\ell/\epsilon)^{\ell}\log C(\pi)}{x^{\frac{1}{4}}}+m+\sqrt{x}+\epsilon x\Big).
\end{align}
By \cite[Prop 5.7]{IK}, there are $\ll \log C(\pi)$ nontrivial zeros $\rho$ of $L(s,\pi)$ with $|\rho|<\frac{1}{4}$.  Thus it follows from \cref{lem:WeightChoice} that
\begin{equation}
\label{eqn:002}
\sum_{\substack{ |\rho|\leq \frac{1}{4} \\ \rho\neq\beta_1}}|\Phi(-\rho\log x)|\ll \sum_{\substack{|\rho|\leq \frac{1}{4}\\ \rho\neq\beta_1}}x^{\frac{1}{4}}\ll x^{\frac{1}{4}}\log C(\pi).
\end{equation}

\begin{lemma}
	\label{lem:sum_zeros}
	Let $\phi$ be defined as in \cref{lem:WeightChoice} with $\epsilon = 8\ell x^{-\frac{1}{8\ell}}$ and $\ell = 4\Cr{LFZDE} m$. If $2\leq C(\pi)^{m}\leq x^{1/(8\Cr{LFZDE})}$ and $\epsilon<\frac{1}{4}$, then
	\[
	\log x\sum_{\substack{ |\rho|\geq \frac{1}{4}\\ \rho\neq\beta_1}}|\Phi(-\rho\log x)|\ll m^2x e^{-\eta_{\pi}(x)/2},\qquad\eta_{\pi}(x):=\inf_{t\geq 3}\Big(\frac{\Cr{ZFR}\log x}{m\log(C(\pi)t)}+\log t\Big).
	\]
\end{lemma}

\begin{proof}
Let $T_0=0$ and $T_j=2^{j-1}$ for all $j\geq 1$.  We consider the sums
	\[
	S_j:=\log x\sum_{\substack{T_{j-1}\leq |\gamma|\leq T_j \\ |\rho|\geq\frac{1}{4},~\rho\neq\beta_1}}|\Phi(-\rho\log x)|.
	\]
	We estimate $|F(-\rho\log x)|$ for $\rho$ in the sum $S_j$ using \cref{lem:WeightChoice} with $\alpha=\ell(1-\beta)$.  Our choices of $\epsilon$ and $\ell$ and our restriction $2\leq C(\pi)^{m}\leq x^{1/(8\Cr{LFZDE})}$ imply that
	\begin{align*}
	|\Phi(-\rho\log x)|\log x\ll \frac{x^{\beta}}{|\rho|}\Big(\frac{2\ell}{\epsilon|\rho|}\Big)^{\ell(1-\beta)}&\ll x T_j^{-\frac{1}{2}}(|\gamma|+3)^{-\frac{1}{2}}x^{-\frac{1-\beta}{2}}(x^{\frac{3}{8}}T_j^{\ell})^{-(1-\beta)}\\
	&\ll x T_j^{-\frac{1}{2}}(|\gamma|+3)^{-\frac{1}{2}}x^{-\frac{1-\beta}{2}} (C(\pi)T)^{-2m\Cr{LFZDE}(1-\beta)}.
	\end{align*}
	By the definition of $\eta_{\pi}(x)$ and \cref{cor:ZFR}, we have the bound
	\[
	(|\gamma|+3)^{-\frac{1}{2}}x^{-\frac{1-\beta}{2}}=e^{-\frac{1}{2}(\log(|\gamma|+3)+(1-\beta)\log x)}\leq e^{-\eta_{\pi}(x)/2}.
	\]
	Consequently, we have that
	\begin{align*}
	S_j \ll \frac{xe^{-\eta_{\pi}(x)/2}}{T_j^{\frac{1}{2}}}\sum_{\substack{T_{j-1}\leq|\gamma|\leq T_j}}(C(\pi)T)^{-2m\Cr{LFZDE}(1-\beta)}\leq \frac{xe^{-\eta_{\pi}(x)/2}}{T_j^{\frac{1}{2}}}\int_0^1(C(\pi)T)^{-2m\Cr{LFZDE}\sigma}dN_m(1-\sigma,T_j).
	\end{align*}
	The Stieltjes integral equals
	\begin{align*}
		(C(\pi)T)^{-2\Cr{LFZDE}m}N_{\pi}(0,T_j)+m\log(C(\pi)T)\int_0^1 (C(\pi)T)^{-2m\Cr{LFZDE}\sigma}N_{\pi}(1-\sigma,T_j)d\sigma,
	\end{align*}
	which we estimate using \cref{prop:LFZDE}.  We conclude that $S_j\ll m^2 x e^{-\eta_{\pi}(x)/2}T_j^{-1/2}$, hence
	\begin{align*}
	\log x\sum_{\substack{ |\rho|\geq \frac{1}{4}\\ \rho\neq\beta_1}}|\Phi(-\rho\log x)|\ll \sum_{j=1}^{\infty}S_j \ll m^2 xe^{-\eta_{\pi}(x)/2}\sum_{j=1}^{\infty}T_j^{-\frac{1}{2}}\ll m^2xe^{-\eta_{\pi}(x)/2},
	\end{align*}
	as desired.
\end{proof}

\begin{lemma}
	\label{lem:eta_bound}
	If $m\geq 1$ and $x\geq 2$, then $e^{-\eta_{\pi}(x)/2}\leq \exp[-\frac{\Cr{ZFR}\log x}{2 m\log C(\pi)}]+\exp[-\frac{\sqrt{\Cr{ZFR}\log x}}{2\sqrt{m}}]$.
\end{lemma}
\begin{proof}
	This is a straightforward optimization problem.
\end{proof}

\begin{proof}[Proof of \cref{prop:PNT}]
Collect the estimates in \cref{lem:smoothing}, \eqref{eqn:001}, \eqref{eqn:002}, \cref{lem:sum_zeros}, and \cref{lem:eta_bound} and apply the prescribed choices for $\ell$ and $\epsilon$ in our range of $x$.
\end{proof}

	\section{Proof of \cref{prop:PNT_Sym,prop:PNT_Sym_2}}
\label{sec:Sym}

Let $f\in S_k^{\mathrm{new}}(\Gamma_0(q))$ be a newform as in the statement of \cref{thm:main_theorem}.  For each prime $p$, let $\theta_p\in[0,\pi]$ be the unique angle such that $a_f(p) = 2\cos\theta_p$.  The modular $L$-function $L(s,f)$ associated to $f$ has the Euler product representation
\[
L(s,f)=\sum_{n=1}^{\infty}\frac{a_{f}(n)}{n^s}=\prod_p\Big(1-\frac{a_f(p)}{p^{s}}+\frac{\chi_0(p)}{p^{2s}}\Big)^{-1},\qquad\re(s)>1,
\]
where $\chi_0$ is the trivial Dirichlet character modulo $q$.  We rewrite the Euler product as
\begin{equation}
\label{eqn:Euler_product}
L(s,f)=\prod_{p|q}\Big(1-\frac{(-\lambda_p p^{-\frac{1}{2}})}{p^{s}}\Big)\prod_{p\nmid q}\prod_{j=0}^1 \Big(1-\frac{e^{i(2j-1)\theta_p}}{p^{s}}\Big)^{-1},\qquad\re(s)>1,
\end{equation}
where $\lambda_p\in\{-1,1\}$ is the eigenvalue of the Atkin--Lehner operator $|_k W(Q_p)$.

\subsection{Standard $L$-functions}

For each $m\geq 1$, we define the Euler product
\begin{equation}
\label{eqn:Euler_product_Lsym}
L(s,\Sym)=\sum_{n=1}^{\infty}\frac{a_{\Sym}(n)}{n^s}=\prod_{p|q}\prod_{j=0}^{m}\Big(1-\frac{\alpha_{j,\mathrm{Sym}^m f}(p)}{p^{s}}\Big)^{-1}\prod_{p\nmid q}\prod_{j=0}^{m} \Big(1-\frac{e^{i(2j-m)\theta_p}}{p^{s}}\Big)^{-1}
\end{equation}
for $\re(s)>1$.  The values $\alpha_{j,\mathrm{Sym}^m f}(p)$ can be determined using \cite[Appendix]{ST}, but an explicit description with uniformity in $f$ can be unwieldy when $q$ is not squarefree.  These explicit descriptions are not germane to our proofs.  We do note that when $f$ corresponds with a non-CM elliptic curve via modularity, a completely explicit and wieldy description of $\alpha_{j,\mathrm{Sym}^m f}(p)$ at $p|q$ can be found in \cite[Appendix]{Chantal}.

We also define
\begin{equation}
\label{eqn:gamma_factor}
L(s,(\Sym)_{\infty})=\begin{cases}
q_{\mathrm{Sym}^m f}\prod_{j=1}^{(m+1) / 2} \Gamma_{\mathbb{C}}(s+(j-\tfrac{1}{2})(k-1))&\mbox{if $m$ is odd,}\\
q_{\mathrm{Sym}^m f}\Gamma_{\mathbb{R}}(s+r) \prod_{j=1}^{m / 2} \Gamma_{\mathbb{C}}(s+j(k-1))&\mbox{if $m$ is even}
\end{cases}
\end{equation}
for a suitable integer $q_{\mathrm{Sym}^m f}$, where $\Gamma_{\mathbb{C}}(s)=\Gamma_{\R}(s)\Gamma_{\R}(s+1)$,  $r=0$ if $m\equiv 0\pmod{4}$, and $r=1$ if $m\equiv 2\pmod{4}$.  Note that $L(\mathrm{Sym}^1 f)=L(s,f)$ and $L(\mathrm{Sym}^0 f)=\zeta(s)$.  One easily checks via \eqref{eqn:Euler_product_Lsym} that
\begin{equation}
\label{eqn:coeffs_1}
a_{\Sym}(p) = U_m(\cos\theta_p),\qquad p\nmid q.
\end{equation}

\begin{theorem}
\label{thm:automorphy}
Let $f\in S_k^{\mathrm{new}}(\Gamma_0(q))$ be as in  \cref{thm:main_theorem}, and let $\pi_f\in\mathfrak{F}_2$ correspond with $f$.  If $m\geq 1$, then $L(s,\Sym)$ is the standard $L$-function associated to the representation $\mathrm{Sym}^m\pi_f\in\mathfrak{F}_{m+1}$, with $L(s,(\mathrm{Sym}^m\pi_f)_{\infty})$ given by \eqref{eqn:gamma_factor}.  We also have $\log q_{\mathrm{Sym}^m f}\ll m\log q$.
\end{theorem}

\begin{proof}
Let $\pi_f$ be the cuspidal automorphic representation of $\GL_2(\A)$ with unitary central character which corresponds with $f$.  Newton and Thorne (\cite[Thm B]{NT} and \cite[Thm A]{NT2}) recently proved that if $m\geq 1$, then the $m$-th symmetric power lift $\mathrm{Sym}^m\pi_f$ is a self-dual cuspidal automorphic representation of $\GL_{m+1}(\mathbb{A})$ with trivial central character whose standard $L$-function is given by \eqref{eqn:Euler_product}.  Moreno and Shahidi \cite{MS2} and Cogdell and Michel \cite[Section 3]{CM} computed $L(s,(\mathrm{Sym}^m \pi_f)_{\infty})$ under the assumption of cuspidality, which we now have.  When $q$ is squarefree, we have $\log q_{\mathrm{Sym}^m f}=m\log q$ \cite[Section 3]{CM}; otherwise, Rouse, following a suggestion of Serre, proved that $\log q_{\Sym}\ll m\log q$ \cite[Section 5]{Rouse}.
\end{proof}

For $m\geq 1$, a straightforward calculation using \eqref{eqn:gamma_factor} and Stirling's formula yields
\begin{equation}
\label{eqn:AC_bound_log}
\log C(\Sym)\ll m\log(kqm),\qquad C(\Sym):=C(\mathrm{Sym}^m\pi_f).
\end{equation}

\subsection{Rankin--Selberg $L$-functions}

Given $f$ as in \cref{thm:automorphy}, let $\pi_f$ be the cuspidal automorphic representation of $\GL_2(\A)$ corresponding to $f$.  Given an integer $m\geq 0$, let $\mathrm{Sym}^m \pi_f$ be the $m$-th symmetric power lift, which is shown in \cref{thm:automorphy} to be a cuspidal automorphic representation of $\GL_{m+1}(\A)$.  (If $m=0$, then $\mathrm{Sym}^m\pi_f=\mathbbm{1}$.)

For $i=1,2$, let $f_i\in S_{k_i}^{\mathrm{new}}(\Gamma_0(q_i))$ be a newform as in \cref{thm:automorphy}, and let $\{\theta_p^{(j)}\}$ be the sequence of Sato--Tate angles for $f_i$.  Suppose that $\pi_{f_1}\neq \pi_{f_2}$.  For integers $m_i\geq 0$, we consider the tensor product $\mathrm{Sym}^{m_1}\pi_{f_1}\otimes\mathrm{Sym}^{m_2}\pi_{f_2}$, whose Rankin--Selberg $L$-function is
\[
L(s,\mathrm{Sym}^{m_1}f_1\times\mathrm{Sym}^{m_2}f_2)\doteq\prod_{p\nmid q_1 q_2}\prod_{j_1=0}^{m_1}\prod_{j_2=0}^{m_2}\Big(1-\frac{e^{i(2j_1-m_1)\theta_p^{(1)}}e^{i(2j_2-m_2)\theta_p^{(2)}}}{p^s}\Big)^{-1}
\]
in view of \cref{thm:automorphy}.  The $\doteq$ suppresses the (more complicated) Euler factors at primes $p|q_1 q_2$ which have an unwieldy (and, for our purposes, unenlightening) explicit description via \cite[Appendix]{ST}.  Instead describing the Euler factors at primes $p|q_{1}q_{2}$ and the gamma factors, we observe that the bound \eqref{eqn:LRS2} applied to the primes $p|q_{1}q_{2}$, while probably very inefficient, is strong enough for us to prove \cref{thm:joint}, and we can estimate $C(\mathrm{Sym}^{m_1}f_1\times\mathrm{Sym}^{m_2}f_2)$ using \eqref{eqn:BH} and \eqref{eqn:AC_bound_log}.  A standard though tedious calculation shows that the Langlands parameters of $L(s,\mathrm{Sym}^{m_1}f_1\times\mathrm{Sym}^{m_2}f_2)$ satisfy the hypotheses of \cref{prop:PNT}; we omit this calculation.  It is straightforward to check that
\begin{equation}
	\label{eqn:coeffs_2}
	a_{\mathrm{Sym}^{m_1}f_1\times\mathrm{Sym}^{m_2}f_2}(p) = U_{m_1}(\cos\theta_p^{(1)})U_{m_2}(\cos\theta_p^{(2)}),\qquad p\nmid q_1 q_2.
\end{equation}

\begin{lemma}
	\label{lem:Rajan}
	If $m_1 m_2\neq 0$ and $f_1\not\sim f_2$ are as in \cref{thm:automorphy}, then $L(s,\mathrm{Sym}^{m_1}f_1\times\mathrm{Sym}^{m_2}f_2)$ extends to an entire function with no pole at $s=1$.
\end{lemma}
\begin{proof}
	This is proved by Harris \cite[Theorem 5.3]{Harris} when $f_1$ and $f_2$ are associated to non-CM elliptic curves.  The proof is identical for other pairs $f_1\not\sim f_2$.  The assumption of Harris's ``Expected Theorems'' is replaced by the automorphy of symmetric powers proved by Newton and Thorne \cite{NT,NT2}.
\end{proof}

\begin{lemma}
	\label{lem:Siegel_1}
	There exists a constant $\Cl[abcon]{ZFR_1}>0$ such that if $m\geq 1$, then $L(s,\Sym)\neq 0$ for
	\[
	\re(s)\geq 1-\frac{\Cr{ZFR_1}}{m^2\log(kqm(3+|\im(s)|))}.
	\]
\end{lemma}
\begin{proof}
When $\im(s)\neq 0$, this follows from \cref{cor:ZFR} and \eqref{eqn:AC_bound_log} (once $\Cr{ZFR_1}$ is made suitably small).  It remains to handle the case where $\im(s)=0$.  Suppose to the contrary that a real zero in this region exists.  Consider the isobaric automorphic representation $\Pi_m=\mathbbm{1}\boxplus\mathrm{Sym}^2\pi_f\boxplus\mathrm{Sym}^m \pi_f$.  Using the identities $\mathrm{Sym}^{m} \pi_f\otimes\mathrm{Sym}^m \pi_f=\mathbbm{1}\boxplus(\boxplus_{j=1}^m\mathrm{Sym}^{2j}\pi_f)$, $\pi_f\otimes\mathrm{Sym}^2\pi_f=\pi_f\boxplus \mathrm{Sym}^3 \pi_f$, and $\mathrm{Sym}^{m} \pi_f\otimes\mathrm{Sym}^2 \pi_f=\boxplus_{j=0}^2\mathrm{Sym}^{m+2-2j}\pi_f$ for $m\geq 2$, we find for $m\geq 1$ that
	\begin{multline*}
L(s,\Pi_m\times\tilde{\Pi}_m)=\zeta(s)^3 L(s,\mathrm{Sym}^{m} f)^4 L(s,\mathrm{Sym}^2 f)^3 L(s,\mathrm{Sym}^4 f) L(s,\mathrm{Sym}^{m+2} f)^2\\
	\times L(s,\mathrm{Sym}^{m-2} f)^2\prod_{j=1}^m L(s,\mathrm{Sym}^{2j} f)
	\end{multline*}
	(with $L(s,\mathrm{Sym}^{m-2}f)^2$ omitted when $m=1$).  The bound $\log C(\Pi_m\times\tilde{\Pi}_m)\ll m^2\log(kqm)$ follows from \eqref{eqn:BH} and \eqref{eqn:AC_bound_log}.  Note that $L(s,\Pi_m\times\tilde{\Pi}_m)$ always has a pole of order exactly 3, but a proposed zero of $L(s,\Sym)$ ensures that $L(s,\Pi_m\times\tilde{\Pi}_m)$ has a real zero of order at least 4 in the region \eqref{eqn:claimed_GHL}.  This contradicts \cref{prop:GHL}, hence no such zero can exist (once $\Cr{ZFR_1}$ is made suitably small). 
\end{proof}

\begin{lemma}
	\label{lem:Siegel_2}
	Let $1\leq m_1,m_2\leq M$. There exists a constant $\Cl[abcon]{ZFR_2}>0$ such that  The Rankin--Selberg $L$-function $L(s,\pi\times\pi')\neq 0$ for
	\[
	\re(s)\geq 1-\frac{\Cr{ZFR_2}}{M^2\log(k_1 q_1 k_2 q_2 M (3+|\im(s)|))}
	\]
	apart from at most one zero $\beta_{m_1,m_2}$.  If $\beta_{m_1,m_2}$ exists, then it is real and simple, and there exist constants $0<\Cl[abcon]{Siegel}<1 $ and $\Cl[abcon]{Siegel2}>0$ such that $\beta_{m_1,m_2}\leq 1-\Cr{Siegel}(k_1 q_1 k_2 q_2 M)^{-\Cr{Siegel2}M^2}$.
\end{lemma}
\begin{proof}
	This follows from \cref{cor:ZFR}, \cref{lem:Siegel}, and \eqref{eqn:AC_bound_log}.
\end{proof}

\subsection{Proofs of \cref{prop:PNT_Sym,prop:PNT_Sym_2}}

\begin{proof}[Proof of \cref{prop:PNT_Sym}]
	This follows from \cref{prop:PNT}(1), \eqref{eqn:coeffs_1}, \cref{thm:automorphy}, \eqref{eqn:AC_bound_log}, and \cref{lem:Siegel_1}.  The conditions in \cref{prop:PNT} are satisfied for $m$ in the claimed range.
\end{proof}

\begin{proof}[Proof of \cref{prop:PNT_Sym_2}]
	This follows from \cref{prop:PNT}(2), \cref{thm:automorphy}, \eqref{eqn:AC_bound_log}, \cref{lem:Rajan}, \eqref{eqn:coeffs_2}, and \cref{lem:Siegel_2}.  The two conditions in \cref{prop:PNT} are satisfied for $m_1$ and $m_2$ in the claimed range.
\end{proof}

\bibliographystyle{abbrv}
\bibliography{GeneralizedLinnik}

\end{document}